\documentclass[11pt]{article}

\usepackage{amsmath,amssymb,amsthm}
\usepackage[dvips]{graphicx}
\usepackage{latexsym}
\usepackage{eucal}
\usepackage[pdftex]{hyperref}
\usepackage[latin1]{inputenc}
\usepackage[english,activeacute]{babel}

\newtheorem{theorem}{Theorem}[section]

\newtheorem{lemma}[theorem]{Lemma}

\theoremstyle{definition}

\newtheorem{remark}[theorem]{Remark}

\newtheorem{example}[theorem]{Example}

\title{{\bf\Large Problems with classic homeomorphisms and three-point boundary conditions}}
\author{{\large  Dionicio Pastor Dallos Santos }\footnote{Email: dionicio@ime.usp.br}\hspace{2mm}
{\bf\large}\vspace{1mm}\\
{\small Department of Mathematics, IME-USP, Cidade Universit\'aria,}\\
{\small  CEP 05508-090, S\~ao Paulo, SP, Brazil}
}

\date{}
\begin{document}
\maketitle

\begin{abstract}
Using Leray-Schauder degree theory  we study the existence of at least one solution for the  boundary  value problem of the type
\[
\left\{\begin{array}{lll}
(\varphi(u' ))'   = f(t,u,u') & & \\
u'(0)=u(0), \  u'(T)= bu'(0), & & \quad \quad 
\end{array}\right.
\] 
where $\varphi:  \mathbb{R}\rightarrow  \mathbb{R}$ is a homeomorphism such that $\varphi(0)=0$,  $f:\left[0, T\right]\times \mathbb{R} \times \mathbb{R}\rightarrow \mathbb{R} $ is a continuous function, and $T$ a positive real number and $b$ some non zero real number.
\end{abstract}
 
\medskip 
  



\section{Introduction}
The purpose of this article is to obtain  some existence results for  the nonlinear boundary value problem of the form
\begin{equation}\label{equa1}
\left\{\begin{array}{lll}
(\varphi(u' ))'   = f(t,u,u') & & \\
u'(0)=u(0), \  u'(T)= bu'(0), 
\end{array}\right.
\end{equation}
where $\varphi:  \mathbb{R}\rightarrow  \mathbb{R}   $ is a homeomorphism such that $\varphi(0)=0$,  $f:\left[0, T\right]\times \mathbb{R} \times \mathbb{R}\rightarrow \mathbb{R} $ is a continuous function, and $T$ a positive real number and $b$ some non zero real number. We call \textsl{solution } of this problem any function $u:\left[0, T\right]\rightarrow \mathbb{R}$ of class $C^{1}$ such that  $\varphi(u')$ is continuously differentiable, satisfying  the  boundary conditions  and $(\varphi(u'(t) ))' = f(t,u(t),u'(t))$ for all $t\in\left[0,T\right]$.

Recently, V. Bouches  and J. Mawhin in \cite{man11} have studied the following boundary value problem: 
\begin{equation}\label{bouches}
\left\{\begin{array}{lll}
(\varphi(u))'   = f(t,u) & & \\
u(T)= bu(0), 
\end{array}\right.
\end{equation}
where   $\varphi:\mathbb{R}\rightarrow (-a,a)$ is a  homeomorphism such that $\varphi(0)=0$,  $f:\left[0, T\right]\times \mathbb{R} \rightarrow \mathbb{R}$ is a continuous function, $a$ and $T$ being positive real numbers and $b$ some non zero real number. The authors obtained the existence of solutions using topological methods based upon  Leray-Schauder degree  \cite{man12}.

Inspired by these results, the main aim of this paper is  to study  the existence of at least one solution for the  boundary  value problem (\ref{equa1}) using Schauder fixed point theorem or Leray-Schauder degree. For this, we reduce the nonlinear boundary value problem to some fixed points problem. The first consequence of this reduction, is that this operator is defined in $C^{1}$. Second, it is completely continuous. Next, adapts a technique introduced by  Ward \cite{wardjr} for the search of a priori bounds for the  possible fixed points required by a Leray-Schauder approach. Such  a  problem does not seem to have been studied in the literature.

The  paper is organized as follows. In Section 2, we establish the notation and terminology used throughout the work. Section 3,  we formulate the fixed point operator equivalent to the problem (\ref{equa1}). Section 4, we give main results in this paper. For these results, we adapt the ideas of  \cite{ man5, man4, man8, ma1, ma, man9, man11} to the present situation.

\section{{Notation and terminology}}
\label{S:-1}
We first introduce some notation. For fixed $T$, we denote  the usual norm in $L^{1}=L^{1}(\left[0,T\right], \mathbb R)$ for $\left\| \cdot \right\|_{L^{1}}$. For $C=C(\left[0,T\right], \mathbb R)$ we indicate the Banach space of all  continuous functions from $\left[0, T\right]$ into $\mathbb R$ witch the norm  $\left\| \cdot \right\|_{\infty}$ and for  $C^{1}=C^{1}(\left[0,T\right], \mathbb R)$  we designate the Banach space of continuously differentiable functions  from $\left[0, T\right]$ into $\mathbb R$ endowed with the usual norm $\left\|u\right\|_{1}=\left\|u\right\|_{\infty} +  \left\|u' \right\|_{\infty}$.

We introduce the following applications:
\medskip 
  
\noindent
the  \textit{Nemytskii operator} $N_f:C^{1} \rightarrow C $, 
\begin{center}
 $N_f (u)(t)=f(t,u(t),u'(t))$, 
\end{center}
the  \textit{integration operator}   $H:C \rightarrow C^{1}$, 
\begin{center}
$ H(u)(t)=\int_0^t u(s)ds$, 
\end{center}
the following continuous linear applications:
\begin{center}
$Q:C \rightarrow C, \ \  Q(u)(t)=\frac{1}{T}\int_0^T u(s)ds$,
\end{center}
\begin{center}
$P:C \rightarrow C, \ \  P(u)(t)=u(0)$,
\end{center}
and finally, we introduced the continuous application
\begin{center}
$B_{\varphi,b}: \mathbb{R} \rightarrow \mathbb{R},  \ \   B_{\varphi,b}(x)=\varphi(bx)-\varphi(x) $.
\end{center}

For  $u\in C$, we write
\begin{center}
$ u_{m}=\displaystyle  \min_{[0,T]} u,  \  u_{M}= \displaystyle \max_{[0,T]}u,  \   u^{+}=\displaystyle \max\left\{ u,0 \right\},  \  u^{-}=\displaystyle \max \left\{ -u,0 \right\}$. 
\end{center}


\section{Fixed point formulations}
\label{S:0}
Let consider the operator
\begin{center}
$M_{1}:C^{1}\rightarrow C^{1}$,

$u\mapsto Q(N_{f}(u))- \frac{B_{\varphi,b}(Pu)}{T}+H\left(\varphi^{-1}\left[\varphi(Pu)+H(N_{f}(u)-Q(N_{f}(u)))+\frac{\Psi B_{\varphi,b}(Pu)}{T}\right]\right)+P(u)$
\end{center}
where $\Psi$ denotes the function which sends  $t$ on $t$ and $\varphi^{-1}$ is understood as the operator  $\varphi^{-1}: C \rightarrow C$ defined for $\varphi^{-1}(v)(t)=\varphi^{-1}(v(t))$. It is clear that $\varphi^{-1}$ is continuous and sends bounded sets into bounded sets.

Using the theorem of Arzel\`a-Ascoli we show that the operator  $M_{1}$ is completely continuous.
\begin{lemma}\label{lema1}
 The operator  $M_{1}:C^{1} \rightarrow C^{1}$ is completely continuous.
\end{lemma}
\begin{proof}
Let $ \Lambda \subset C^{1}$ be a bounded set. Then, if $u \in \Lambda$, there exists a constant  $\rho>0$ such  that  
\begin{equation}\label{dp4}
\left\| u\right\|_{1}\leq \rho.
\end{equation}
Next, we show that  $\overline{M_{1}(\Lambda)}\subset  C^{1}$ is a compact set. Let $(v_{n})_{n} $  be a sequence in   $M_{1}(\Lambda)$,  and 
 let $(u_{n})_{n}$  be a sequence  in  $\Lambda$ such that  $v_{n}=M_{1}(u_{n})$. Using (\ref{dp4}), we have that there exists a constant  $L_{1}>0$ such that,  for all $n\in \mathbb{N}$,   
\begin{center}
 $\left\| N_{f}(u_{n})\right\|_{\infty}\leq L_{1}$,
\end{center}
which implies that
\begin{center}
$\left\| H(N_{f}(u_{n})-Q(N_f(u_{n}))) \right\|_{\infty}\leq 2L_{1}T$.
\end{center}
Hence the sequence  $ \left(H(N_{f}(u_{n})-Q(N_f(u_{n})))\right)_{n}$ is bounded in  $C$. Moreover, for  $t, t_{1}\in\left[0, T\right]$ and for all $n\in \mathbb{N}$, we have that 
\begin{align*}
&\left| H(N_{f}(u_{n})-Q(N_f(u_{n})))(t) - H(N_{f}(u_{n})-Q(N_f(u_{n})))(t_{1})  \right|\\
&\leq \left|\int_{t_{1}}^{t}N_f(u_{n})(s) ds\right| + \left| \int_{t_{1}}^{t}Q(N_f(u_{n}))(s) ds\right|\\
&\leq L_{1}\left|t-t_1\right| +  \left|t-t_1\right|\left\| Q(N_f(u_{n}))\right\|_{\infty}\\
&\leq 2L_{1}\left|t-t_1\right|,
\end{align*} 
which  implies that  $ \left(H(N_{f}(u_{n})-Q(N_f(u_{n})))\right)_{n}$ is equicontinuous. Thus, by the Arzel\`a-Ascoli theorem there is a subsequence of  $\left(H(N_{f}(u_{n})-Q(N_f(u_{n})))\right)_{n}$, which we call $\left(H(N_{f}(u_{n_{j}})-Q(N_f(u_{n_{j}})))\right)_{j}$, which is  convergent in $C$. Then, passing to a subsequence if  necessary, we obtain that the sequence $$ \left(H(N_{f}(u_{n_{j}})-Q(N_f(u_{n_{j}}))) + \frac{\Psi B_{\varphi,b}(Pu_{n_{j}})}{T} +  \varphi(P(u_{n_{j}}))\right)_{j}$$
is convergent in  $C$. Using the fact that   $\varphi^{-1}: C \rightarrow  C$ is continuous it follows from
\begin{center}
$M_{1}(u_{n_{j}})'=\varphi^{-1}\left[\left(H(N_{f}(u_{n_{j}})-Q(N_f(u_{n_{j}}))) + \frac{\Psi B_{\varphi,b}(Pu_{n_{j}})}{T}  + \varphi(P(u_{n_{j}}))\right)       \right]$
\end{center}
that the sequence  $(M_{1}(u_{n_{j}})')_{j}$ is convergent in  $C$. Therefore, passing if necessary to a subsequence, we have that  $(v_{n_{j}})_{j}=( M_{1}(u_{n_{j}}))_{j}$  is convergent in $C^{1}$. Finally, let  $(v_{n})_{n}$ be a sequence in  $\overline{M_{1}(\Lambda)}$. Let  $(z_{n})_{n}\subseteq M_{1}(\Lambda)$  be such that
\[
\lim_{n \to \infty}\left\| z_{n}-v_{n}\right\|_{1}=0.
\]
Let $(z_{n_{j}})_{j}$ be  a subsequence of  $(z_{n})_{n}$  such that  converge to $z$. It follows that  $z\in \overline{M_{1}(\Lambda)}$ and  $(v_{n_{j}})_{j}$ converge to $z$. This concludes the proof.
\end{proof}
\begin{lemma}\label{mate2}
$u \in C^{1}$ is  a solution  of (\ref{equa1}) if and only if  $u$ is a fixed point of the operator $M_{1}$.
\end {lemma} 
\begin{proof}
Let $u\in C^{1}$, we have the following equivalences:
\medskip
 
$(\varphi(u'))'  = N_f (u), \  u'(T)=bu'(0), \  u'(0)=u(0)$
\medskip 
 
\medskip 
  
\noindent
$ \Leftrightarrow (\varphi(u'))'  = N_f (u)- \left(Q(N_{f}(u))- \frac{ B_{\varphi,b}(u'(0))}{T}\right)$,    

$Q(N_{f}(u))- \frac{ B_{\varphi,b}(u'(0))}{T}=0, \  u'(0)=u(0)$
\medskip 
 
\medskip 
  
\noindent
$ \Leftrightarrow \varphi(u')=H \left(N_f (u)-Q(N_{f}(u))\right)+ \frac{\Psi B_{\varphi,b}(u'(0))}{T} + \varphi(u'(0))$,
  
$Q(N_{f}(u))- \frac{ B_{\varphi,b}(u'(0))}{T}=0, \  u'(0)=u(0)$ 
\medskip 
 
\medskip 
  
\noindent   
$\Leftrightarrow u'= \varphi^{-1}\left[H( N_f (u)-Q(N_{f}(u)))+ \frac{\Psi B_{\varphi,b}(u'(0))}{T}+\varphi(u'(0))\right]$,
 
$Q(N_{f}(u))- \frac{ B_{\varphi,b}(u'(0))}{T}=0, \  u'(0)=u(0)$
\medskip 
 
\medskip 
  
\noindent   
$\Leftrightarrow u= H\left(\varphi^{-1}\left[H( N_f (u)-Q(N_{f}(u)))+ \frac{\Psi B_{\varphi,b}(u'(0))}{T}+\varphi(u'(0))\right]\right)+u(0)$,
 
$Q(N_{f}(u))- \frac{ B_{\varphi,b}(u'(0))}{T}=0,  \   u'(0)=u(0)$   
\medskip 
 
\medskip 
  
\noindent    
$\Leftrightarrow u= H\left(\varphi^{-1}\left[H( N_f (u)-Q(N_{f}(u)))+ \frac{\Psi B_{\varphi,b}(u(0))}{T}+\varphi(u(0))\right]\right)+u(0)$,
 
 $Q(N_{f}(u))- \frac{ B_{\varphi,b}(u(0))}{T}=0 $     
\medskip 
 
\medskip 
  
\noindent
$\Leftrightarrow u= Q(N_{f}(u)) - \frac{B_{\varphi,b}(Pu)}{T} + H\left(\varphi^{-1}\left[H( N_f (u)-Q(N_{f}(u))) + \frac{\Psi B_{\varphi,b}(Pu)}{T} + \varphi(Pu)\right]\right) + Pu$.
\end{proof}

\begin {remark}
 Note that  $u'(T)= bu'(0) \Leftrightarrow  Q(N_{f}(u))= \frac{ B_{\varphi,b}(u'(0))}{T}$.
\end {remark}
In order to apply Leray-Schauder degree to the operator $M_{1}$, we  introduced a family of problems  depending on a parameter $\lambda$. We remember  that to each  continuous function  $f:\left[0, T\right]\times \mathbb{R} \times \mathbb{R}\rightarrow \mathbb{R}$ we associate 
its Nemytskii operator $N_{f}: C^{1}\rightarrow C$ defined by
\begin{center}
$N_{f}(u)(t) = f(t, u(t), u'(t))$.
\end{center}
For  $ \lambda \in [0,1]$, we consider the family of boundary value problems
\begin{equation}\label{misto3}
\left\{\begin{array}{lll}
(\varphi(u' ))'   = \lambda N_{f}(u)+(1-\lambda)Q(N_{f}(u)) & & \\
u'(0)=u(0), \  u'(T)=bu'(0). 
\end{array}\right.
\end{equation}
 Notice that (\ref{misto3}) coincide with (\ref{equa1}) for $\lambda =1$. So, for  each $ \lambda \in [0,1]$, the operator associated to  \ref{misto3} by Lemma \ref{mate2}  is the  operator $M(\lambda,\cdot)$, where $M$  is defined on $[0,1]\times C^{1}$  by 
\begin{center}
$ M(\lambda,u)= Q(N_{f}(u))- \frac{B_{\varphi,b}(Pu)}{T}+ H\left(\varphi^{-1}\left[\varphi(Pu)+ \lambda H(N_{f}(u)-Q(N_{f}(u)))+\frac{\Psi B_{\varphi,b}(Pu)}{T}\right]\right)+P(u)$.
\end{center}
Using the same arguments as in the proof  of Lemma \ref{lema1}  we show that the operator  $M$ is completely continuous. Moreover, using the same reasoning as above,  the system (\ref{misto3}) (see Lemma \ref{mate2}) is equivalent to the problem 
\begin{equation}\label{igualdade}
u=M(\lambda, u).
\end{equation}
In order to prove the existence at least one solution of (\ref{equa1}) we consider the family of problems
\begin{equation}\label{equazeta}
\left\{\begin{array}{lll}
(\varphi(u' ))'   = \lambda Q(N_{f}(u)) & & \\
\int_0^T f(t,u(t),u'(t))dt= \varphi(bu(0))-\varphi(u(0)), \ u'(0)=u(0). 
\end{array}\right.
\end{equation}
 On the other hand, we consider the homotopy
\begin{equation}\label{equaope}
 Z(\lambda,u)=P(u)+Q(N_{f}(u)) -\frac{B_{\varphi,b}(Pu)}{T} + H\left(\varphi^{-1} \left[\lambda \frac{\Psi B_{\varphi,b}(Pu)}{T}+\varphi(Pu)\right]\right),
\end{equation}
where  $Z(1,\cdot)=M(0,\cdot)$. By the same  argument as above, the 
operator $Z:\left[0,1\right]\times C^{1}\rightarrow C^{1}$  (see Lemma \ref{lema1}) is  completely continuous.
\begin{lemma}\label{soluciz}
 If  $(\lambda, u)\in\left[0,1\right] \times C^{1}$ is such that $u=Z(\lambda,u)$, then  $u$ is  solution of  (\ref{equazeta}).
\end{lemma}
\begin{proof}
Let $(\lambda, u)\in\left[0,1\right] \times C^{1}$ be such that $u=Z(\lambda,u)$. It follows that 
\begin{center}
$\int_0^T f(t,u(t),u'(t))dt= \varphi(bu(0))-\varphi(u(0))$
\end{center}
and
\begin{equation}\label{equa2}
u'(t) = \varphi^{-1}\left[t\lambda \frac{\varphi(bu(0))- \varphi(u(0))}{T}+ \varphi(u(0)) \right]
\end{equation}
for all  $t\in \left[0, T\right]$. Applying  $\varphi$ to both members and differentiating, we deduce that  
\begin{center}
$(\varphi(u'(t)))' = \lambda \frac{\varphi(bu(0))- \varphi(u(0))}{T} = \lambda Q(N_{f}(u)) $
\end{center}
 for all $t\in \left[0, T\right]$. 

On the other hand, using (\ref{equa2}) for   $t=0$, we obtain  $u'(0)=u(0)$. This completes the  proof. 
\end{proof}


\section{Main results}
\label{S:2}
In this section, we present and prove our main results. These results are inspired on works by   Bereanu and  Mawhin \cite{ma} and  Man\'asevich and  Mawhin \cite{man}. We denote by $deg_{B}$ the Brouwer degree and  for $deg_{LS}$ the 
Leray-Schauder degree, and define the mapping $G:\mathbb{R}^{2}\rightarrow \mathbb{R}^{2}$ by 
\begin{equation}\label{gato}
G:\mathbb{R}^{2}\rightarrow \mathbb{R}^{2}, \  \quad (x,y)\mapsto \left( \frac{B_{\varphi, b}(x)}{T}-\frac{1}{T}\int_0^T f(t, x+yt, y)dt, -x+y\right).
\end{equation}
\begin{theorem}\label{teoemaprincipalmisto}
Assume that  $\Omega$ is an open bounded set in  $C^{1}$ such that the following conditions hold.
\begin{enumerate}
\item If $(\lambda , u) \in \left[0,1\right] \times C^{1}$ is such that $u=Z(\lambda, u)$, then $u\notin \partial \Omega$.
\item  The Brouwer degree	
\begin{center}
$deg_{B}(G,\Omega \cap  \mathbb{R}^{2},0)\neq 0$,
\end{center}
where we consider the natural identification $(x,y)\approx x+yt$ of  $\mathbb R^{2}$ with related functions in $C^{1}$.
\item For each $\lambda \in (0, 1]$ the problem  (\ref{misto3})  has no solution on $\partial\Omega$.
\end{enumerate}
Then  (\ref{equa1}) has a solution.
\end{theorem} 
\begin{proof}
Using  hypothesis 1 and that $Z$ is completely continuous, we deduce that for each $\lambda \in \left[0,1\right]$, the Leray-Schauder 
degree $deg_{LS}(I-Z(\lambda,\cdot),\Omega,0)$ is well-defined, and by the homotopy invariance  imply that
\begin{center}
$deg_{LS}(I-Z(1 ,\cdot),\Omega,0)= deg_{LS}(I-Z(0,\cdot),\Omega,0)$.
\end{center}

On the other hand, we have 
\begin{center}
$deg_{LS}(I-Z(0,\cdot),\Omega,0) = deg_{LS}(I-(P+QN_{f}- \frac{B_{\varphi,b}P}{T}+HP) ,\Omega,0)$.
\end{center}
But the range of the mapping
\begin{center}
$u\longrightarrow P(u)+Q(N_{f}(u))- \frac{B_{\varphi,b}(P(u))}{T}+H(P(u))$
\end{center}
is contained in the subspace of related functions, isomorphic to $\mathbb{R}^{2}$.  Thus, using a reduction property of  Leray-Schauder degree \cite{man7, man12}
\begin{align*}
& deg_{LS}(I-(P+QN_{f}- \frac{B_{\varphi,b}P}{T}+HP) ,\Omega,0)\\
& = deg_{B}\left(I-(P+QN_{f}- \frac{B_{\varphi,b}P}{T}+HP)\left|_{\overline{\Omega \cap \mathbb{R}^{2}}}\right. ,\Omega \cap \mathbb{R}^{2}, 0\right)\\
&=deg_{B}(G,\Omega \cap \mathbb{R}^{2},0)\neq 0.
\end{align*}

On the other hand, using the  fact that $M$ is completely continuous,  that $Z(1,\cdot)$  coincides with the operator  $M(0,\cdot)$ and the hypothesis 3, we deduce  that for each $\lambda \in \left[0,1\right]$, \  $deg_{LS}(I-M(\lambda,\cdot),\Omega,0)$ is well-defined, and by the homotopy invariance we have 
\begin{center}
$deg_{LS}(I-M(1,\cdot),\Omega,0)= deg_{LS}(I-M(0,\cdot),\Omega,0)$.
\end{center}
Hence,  $deg_{LS}(I-M(1,\cdot),\Omega,0)\neq 0$. This, in turn, implies that there exists  $u\in\Omega$  such that  $M_{1}(u)=u$, which is a solution for (\ref{equa1}).
\end {proof}
The problem (\ref{equa1}) can be studied by placing some special conditions on $f(t,x,y)$.
 \begin{theorem}\label{perro} 
 Assume that the following conditions hold  for a opportune  $\rho > 0$.  
\begin{enumerate}
\item  There exists a function  $h\in C$ such that
\begin{center}
$\left|f(t,x,y)\right|\leq h(t) \  for  \  all  \  (t,x,y)\in \left[0,T\right]\times \mathbb{R}^{2}$.
\end{center}
\item There exists $M_{1}<M_{2}$  such that for all  $u\in C^{1}$,
\begin{center}
$\int_0^T f(t,u(t),u'(t))dt - B_{\varphi,b}(u'(0))\neq 0$ \ se \ $u_{m}'\geq M_{2} $,
\end{center}
\begin{center}
$\int_0^T f(t,u(t),u'(t))dt - B_{\varphi,b}(u'(0))\neq 0$ \ se \ $u_{M}'\leq M_{1} $.
\end{center}
\item The Brouwer degree	
\begin{center}
$deg_{B}(G,B_{\rho}(0) \cap \mathbb{R}^{2},0)\neq 0$.
\end{center}
\end{enumerate}
Then problem (\ref{equa1}) has  at least one solution.
\end{theorem} 
\begin{proof} 
 Let $(\lambda,u)\in \left[0,1\right]\times C^{1}$ be such that $u$ is a solution of (\ref{misto3}). Using (\ref{igualdade}), we have that
 \begin{center}
$u= M(\lambda, u ) = Q(N_{f}(u))- \frac{B_{\varphi,b}(Pu)}{T}+H\left(\varphi^{-1}\left[\varphi(Pu)+ \lambda H(N_{f}(u)-Q(N_{f}(u)))+\frac{\Psi B_{\varphi,b}(Pu)}{T}\right]\right)+P(u)$.
\end{center}
 By evaluation of $u$ at $0$, we obtain 
\begin{center}
$\int_0^T f(t,u(t),u'(t))dt - B_{\varphi,b}(u(0))= 0$.
\end{center}
 Differentiating $u$ and using the fact that $u'(0)=u(0)$, we deduce that  
\begin{center}
$\int_0^T f(t,u(t),u'(t))dt - B_{\varphi,b}(u'(0))= 0$.
\end{center}
 Now by hypothesis 2 it follows that 
 \begin{center}
$u_{m}'<M_{2}$ \ and \  $u_{M}'>M_{1}$. 
\end{center}
Then, there exists $\omega \in [0,T]$ such that  $M_{1}<u'(\omega)<M_{2}$. Moreover,
\begin{center}
$\int_\omega ^t (\varphi(u'(s)))'ds= \lambda \int_\omega ^t  N_{f}(u)(s)ds +(1-\lambda) \int_\omega ^t Q( N_{f}(u))(s)ds$
\end{center}
for all $t\in \left[0, T\right]$. By hypothesis 1 it follows that 
\begin{center}
$\left| \varphi(u'(t))\right| \leq \left| \varphi(u'(\omega))\right|+ 2\left\|h\right\|_{L^{1}} <  L + 2\left\|h\right\|_{L^{1}} $,
\end{center}
where  $L= $max$ \left\{ \left| \varphi(M_{2})\right|, \left| \varphi(M_{1})\right| \right\}$. Hence, 
\begin{center}
$\left\|u'\right\|_{\infty}< a$,
\end{center}
where  $a= $max$\left\{\left|\varphi^{-1}(L+2\left\|h\right\|_{L^{1}})\right|, \left|\varphi^{-1}(-L-2\left\|h\right\|_{L^{1}})\right|\right\}$. Using the fact that  $u'(0)=u(0)$, we obtain 
\begin{center}
$\left|u(t)\right|\leq \left|u(0)\right|+ \int_0^T \left|u'(s)\right|dt < a+aT  \  \   (t\in [0,T])$,
\end{center}
and hence
\begin{center}
$\left\|u\right\|_{1}=\left\|u\right\|_{\infty}+\left\|u'\right\|_{\infty}< a+aT+a=a(2+T)= R_{1}$.
\end{center} 
  
Let $(\lambda,u)\in \left[0,1\right]\times C^{1}$ be such that $u=Z(\lambda,u)$. Using Lemma \ref{soluciz}, $u$ is a solution of (\ref{equazeta}), which implies that 
\begin{center}
$\int_0^T f(t,u(t),u'(t))dt - B_{\varphi,b}(u'(0))= 0$.
\end{center} 
 Using hypothesis 2 it follows that there exists $\tau \in [0,T]$ such that  $M_{1}<u'(\tau)<M_{2}$. Moreover,
 \begin{center}
$\left| \varphi(u'(t))\right| \leq \left| \varphi(u'(\tau))\right|+  \left|\lambda \int_\tau ^t  Q(N_{f}(u))(s)ds\right| $
\end{center}
for all $t\in \left[0, T\right]$. Now by hypothesis 1 it follows that
\begin{center}
$\left| \varphi(u'(t))\right| <  L + \left\|h\right\|_{L^{1}} $.
\end{center}
Hence, 
\begin{center}
$\left\|u'\right\|_{\infty}< b$,
\end{center}
where  $b=$max$\left\{\left|\varphi^{-1}(L+\left\|h\right\|_{L^{1}})\right|, \left|\varphi^{-1}(-L-\left\|h\right\|_{L^{1}})\right|\right\}$.  Now for  $t\in [0,T]$
\begin{center}
$\left|u(t)\right|\leq \left|u(0)\right|+ \int_0^T \left|u'(s)\right|dt < b+bT $,
\end{center}
and hence
\begin{center}
$\left\|u\right\|_{1}=\left\|u\right\|_{\infty}+\left\|u'\right\|_{\infty}< b+bT+b=b(2+T)= R_{2}$.
\end{center} 
Defining $\Omega=B\rho(0) $ in Theorem \ref{teoemaprincipalmisto}, where $B\rho(0)$  is the open ball in $C^{1}$ center $0$ and radius $\rho\geq $max$\left\{R_{1}, R_{2}\right\}$, we can guarantee the existence of at least a solution of (\ref{equa1}).
\end{proof}
 
Using the same arguments as in the proof of Theorem \ref{perro} we obtain the following existence result.
\begin{theorem}\label{impar} 
 Let $\varphi$  be an odd homeomorphism. Assume that the following conditions hold.  
\begin{enumerate}
\item  There exists a function  $h\in C$ such that
\begin{center}
$\left|f(t,x,y)\right|\leq h(t) \  for  \  all  \  (t,x,y)\in \left[0,T\right]\times \mathbb{R}^{2}$.
\end{center}
\item The Brouwer degree	
\begin{center}
$deg_{B}(G,B_{\rho}(0) \cap \mathbb{R}^{2},0)\neq 0, \ for \  a \  opportune  \  \rho > 0$.
\end{center}
\end{enumerate}
Then problem (\ref{equa1}) with $b=-1$  has  at least one solution.
\end{theorem} 
 
In the next lemma, we  adapt the ideas of Ward \cite{wardjr} to obtain the required a priori bounds. 
\begin{lemma}\label{ward}
Assume that $f$ satisfies the following conditions.
\begin{enumerate}
\item There exists  $c \in C$ such that 
\begin{center}
$f(t,x,y)\geq c(t)$
\end{center}
for all $(t,x,y)\in [0, T]\times \mathbb{R}\times \mathbb{R}$.
\item There exists  $M_{1}<M_{2}$  such that for all  $u\in C^{1}$,
\begin{center}
$\int_0^T f(t,u(t),u'(t))dt \neq 0$ \ if \ $u_{m}'\geq M_{2} $,
\end{center}
\begin{center}
$\int_0^T f(t,u(t),u'(t))dt \neq 0$ \ if \ $u_{M}'\leq M_{1}$.
\end{center}
\end{enumerate}
If $b=1$ and $(\lambda, u)\in \left[0, 1\right]\times C^{1}$ is such that $u=M(\lambda,u)$, then $$\left\|u'\right\|_{\infty}< r,$$ where  $$r= \max \left\{ \left|\varphi^{-1}\left(L+2\left\|c^{-}\right\|_{L^{1}}\right)\right|,  \left|\varphi^{-1}\left(-L-2\left\|c^{-}\right\|_{L^{1}}\right)\right|\right\},  L=\max \left\{ \left| \varphi(M_{2})\right|, \  \left| \varphi(M_{1})\right| \right\}.$$
\end{lemma}
\begin{proof}
Use the same arguments as in the proof of Theorem \ref{perro}  and  the following inequality  $\left|f(t,u(t),u'(t))\right|\leq f(t,u(t),u'(t)) + 2c^{-}(t)$  for all $t\in \left[0,T\right]$.
\end{proof}

Now we can prove  an existence theorem for (\ref{equa1}).
\begin{theorem}\label{exem} 
 Let  $f$ be continuous and satisfy conditions (1) and (2) of Lemma \ref{ward}. Assume that the following conditions hold for some $\rho \geq r(2+T)$.  
\begin{enumerate}
\item  The equation
\begin{center}
$ G(x,y)=(0,0)$,
\end{center}
has no solution on  $\partial B_{\rho}(0)\cap\mathbb{R}^{2}$, where we consider the natural identification $(x,y)\approx x+yt$ of  $\mathbb R^{2}$ with related functions in $C^{1}$.
\item The Brouwer degree	
\begin{center}
$deg_{B}(G,B_{\rho}(0) \cap \mathbb{R}^{2},0)\neq 0$,
\end{center}
\end{enumerate}
Then problem (\ref{equa1}) with $b=1$ has a solution.
\end{theorem}  
\begin{proof}
If $b=1$ and  $ (\lambda, u) \in \left[0,1\right]\times C^{1}$ is such that $u=Z(\lambda,u)$,  by evaluation of $u$ at $0$, we have  that
\begin{equation}\label{oso}
\int_0^T f(t,u(t),u'(t))dt=0.
\end{equation}
Moreover, $u$  is a function of the form $u(t)=x + yt, \ y=x $. Thus, By (\ref{oso})
\begin{center}
$\int_0^T f(t,x+yt,y)dt=0$,
\end{center}
which, together with hypothesis 1, implies that $u= x+tx \notin \partial B_{\rho}(0)$.

Let $b=1$ and $(\lambda,u)\in \left[0,1\right]\times C^{1}$ be such that $u=M(\lambda,u)$. Using Lemma \ref{ward}, we have  that $\left\|u\right\|_{1}<r(2+T)$. Thus we have proved that (\ref{misto3}) has no solution in $\partial B_{\rho}(0)$ for $b=1$ and $(\lambda,u)\in \left[0,1\right]\times C^{1}$, hence the conditions of Theorem \ref{teoemaprincipalmisto} are satisfied, the proof is complete.
\end{proof}

Our next theorem is a generalization of Theorem \ref{impar}. We need first of the following applications.
The \textit{differential operator} 
\begin{center}
$D: {\rm dom}(D) \rightarrow C, \ \  u \mapsto u'$,
\end{center}
onde ${\rm dom}(D)=\left\{ u\in C^{1}_{b}:\varphi(u')\in C^{1}\right\},\  C^{1}_{b}= \left\{u\in C^{1}:u'(T)= bu'(0), \ \  u'(0)= u(0)\right\}$.

The operator
\begin{center}
$D_{\varphi}: {\rm dom}(D_{\varphi}) \rightarrow C,  \ \  u \mapsto (\varphi(u))'$,
\end{center}
onde ${\rm dom}(D_{\varphi})=\left\{u\in C:\varphi(u)\in C^{1}\right\}$.
 
The operator
\begin{center}
$\widetilde{D}_{\varphi}=D_{\varphi} D:{\rm dom}(D)\rightarrow C, \ \  u \mapsto (\varphi(u'))'$.
\end{center}

  When $b<0$,  $-\varphi(\cdot)$  and   $\varphi(b  \ \cdot)$  are simultaneously  increasing or  decreasing. In this case,  $B_{\varphi,b}(\cdot)=  \varphi(b  \ \cdot)  - \varphi(\cdot) $ is injective. Thus, the operator $\widetilde{D}_{\varphi}$ has an inverse given by
\begin{center}
$u \mapsto H\left(\varphi^{-1} \left[ \varphi\left( B^{-1}_{\varphi,b} \left(\int_0^T u(s)ds\right)\right) + \int_0^t u(s)ds \right] \right)  +   B^{-1}_{\varphi,b}\left( \int_0^T u(s)ds\right)$.
\end{center}
Hence,
\begin{align*}
&(\varphi(u'))'  = N_f (u), \ u'(T)=bu'(0), \  u'(T)=bu'(0)   \\
&\Leftrightarrow (D_{\varphi}D)(u)=N_{f}(u),\  u\in {\rm dom}(D) \\
&\Leftrightarrow u = (D_{\varphi}D)^{-1}N_{f}(u), \  u\in C^{1}.
\end{align*}
Hence our problem is finding a fixed point of the operator
\begin{center}
$ \Gamma:= (D_{\varphi}D)^{-1}N_f : C^{1}  \rightarrow {\rm dom}(D)$.

$u \mapsto H\left(\varphi^{-1} \left[\varphi\left( B^{-1}_{\varphi,b} \left(\int_0^T N_f (u)(s) ds\right)\right) + \int_0^t N_f (u)(s)ds\right] \right) +   B^{-1}_{\varphi,b}\left( \int_0^T N_f (u)(s)ds\right)$.
\end{center}

In the next theorem, we adapt the ideas of Bouches and Mawhin \cite{man11} to obtain the existence of at least one solution of (\ref{equa1}).
\begin{theorem}\label{lunes}
If there exists a function  $h\in C$ such that
\begin{center}
$\left|f(t,x,y)\right|\leq h(t)$
\end{center}
for all  $(t,x,y)\in \left[0, T\right]\times \mathbb{R}^{2}$, then problem (\ref{equa1})  with $b<0$  has a solution.
\end{theorem}
\begin{proof}
Let us consider $v= \Gamma(u):= (D_{\varphi}D)^{-1}N_f(u)$. Then, 
\begin{center}
$v'(T)= bv'(0), \ v'(0)=v(0)$
\end{center}
and
\begin{center}
$N_{f}(u)= (D_{\varphi}D)(v)=(\varphi(v'))'$.
\end{center}
Because $v\in C^{1}$   is such  that $v'(T)= bv'(0) $, there exists  $\tau \in [0, T]$ such that  $v'(\tau)=0$, which implies  $\varphi(v'(\tau))=0$ and  
\begin{center}
$\left|\varphi(v'(t))\right|=\left|\int_\tau ^t (\varphi (v'(s)))'ds\right| \leq \int_\tau^t \left|N_f(u)(s)\right|ds \leq \int_ 0^T \left|f(s,u(s),u'(s))\right|ds \leq \left\|h\right\|_{L^{1}}  \  \  (t\in [0,T])$,
\end{center}
 which implies
\begin{center}
$\left\|v'\right\|_{\infty}\leq \beta$,
\end{center}
where $\beta = $max$ \left\{\left|\varphi^{-1}\left(\left\|h\right\|_{L^{1}}\right)\right|,  \   \left|\varphi^{-1}\left(-\left\|h\right\|_{L^{1}}\right)\right| \right\}$. Using the fact that  $v'(0)=v(0)$, we deduce that
\begin{center}
$\left| v(t)\right|\leq \left|v(0)\right|+\int_0^t \left |v'(s)\right|ds\leq +\left|v(0)\right|+ \int_0^T \left|v'(s)\right|ds\leq \beta + \beta T $
\end{center}
for all $t\in [0,T]$, and hence 
\begin{center}
$\left\|v\right\|_{1}=\left\|v\right\|_{\infty}+\left\|v'\right\|_{\infty}\leq \beta + \beta T + \beta= \beta(2+T)$.
\end{center}
Because the  $\Gamma $  is completely continuous and bounded, we can  use Schauder's Fixed Point Theorem  to deduce  the existence of at least one fixed point in $\overline{B_{\beta(2+T)}(0)}$. The proof  is complete. 
\end{proof}

Let us give now an application of Theorem \ref{exem}.
\begin{example}
Let us consider the problem 
\begin{equation}\label{exemplo}
\left((u')^{3}\right)' = \frac{e^{u'}}{2} -1, \quad u(0)=u'(0)=u'(T),
\end{equation}
for  $M_{1}=-1, \  M_{2}=1 ,  \  \rho\geq (1+2T)^{1/3}(2+T)$  and  $c(t)=-1$ for all $t\in \left[0, T\right]$. So, problem (\ref{exemplo}) has at least one solution. 
\end{example}

\section*{Acknowledgements}
 This research was supported by CAPES and CNPq/Brazil.

\bibliographystyle{plain}
\renewcommand\bibname{References Bibliogr\'aficas}

\end{document}